\documentclass[reqno,11pt]{amsart}
\usepackage{txfonts}

\usepackage{amssymb}
\usepackage{amsthm}
\usepackage{amsmath}
\usepackage{mathrsfs}

\theoremstyle{theorem}
\newtheorem{theorem}{Theorem}[section]

\newtheorem{cla}{Claim}

\newtheorem{thm}{Theorem}

\theoremstyle{definition}
\newtheorem{definition}[theorem]{Definition}
\newtheorem{remark}[theorem]{Remark}

\newcommand{\C}{\mathbb{C}}

\newcommand{\R}{\mathbb{R}}

\newcommand{\HS}{\mathcal{H}}
\newcommand{\BH}{\mathcal{L}(\mathcal{H})}
\newcommand{\D}{\mathcal{D}}
\newcommand{\RA}{\mathcal{R}}

\title[Decomposition of Normal Operators]{{\Large Decomposition of Normal Operators and Its Application to Spectral Theorem}}

\author[Katsukuni Nakagawa]{}



\begin{document}

\maketitle

\centerline{KATSUKUNI\ \ NAKAGAWA}
\medskip
 \centerline{Graduate School of Advanced Science and Engineering, Hiroshima University}
   \centerline{Kagamiyama 1-3-1, Higashi-Hiroshima 739-8526, Japan}
\centerline{e-mail: ktnakagawa@hiroshima-u.ac.jp}


\vspace{2\baselineskip}
\centerline{\scshape Abstract}
\medskip
\begin{center}
\begin{minipage}{0.8\textwidth}
{\normalsize A decomposition theorem for self-adjoint operators proved by Riesz and Lorch is extended to normal operators.
This extension gives a new proof of the spectral theorem for unbounded normal operators.}
\end{minipage}
\end{center}
\medskip

\section{Introduction}

Let $\HS$ be a Hilbert space. 
We denote by $\BH$ the set of all bounded linear operators on $\HS$.
Here is a decomposition theorem for self-adjoint operators proved by Riesz and Lorch in \cite{Riesz-Lorch}.

\begin{thm}[\cite{Riesz-Lorch}]
\label{thm:092}
Let $T:\HS\to\HS$ be a self-adjoint linear operator. 
Then there are $A,B\in\BH$ such that the following four assertions hold:
\begin{enumerate}
\item[(i)]$A$ is injective.
\item[(ii)]$T=A^{-1}B$.
\item[(iii)]$AB=BA$.
\item[(iv)]$A$ and $B$ are self-adjoint.
\end{enumerate}
\end{thm}

The first aim of this paper is to extend Theorem \ref{thm:092} to normal operators as follows:

\begin{theorem}
\label{thm:018}
Let $T:\HS\to\HS$ be a normal linear operator. 
Then there are $A,B\in\BH$ such that (i)$\sim$(iii) in Theorem \ref{thm:092} and the following assertion hold:
\begin{center}
{\rm (iv)} $A$ is self-adjoint and $B$ is normal. 
Moreover, if $T$ is self-adjoint, then so is $B$.
\end{center}
\end{theorem}

To explain the second aim of this paper, we recall the next spectral theorem for normal operators:

\begin{thm}
\label{thm:065}
Let $T:\HS\to\HS$ be a normal linear operator. 
Then there is a Borel spectral measure $E$ such that the following equality holds:
\[
T=\int_{\C}z\,E(dz).
\]
\end{thm}

In \cite{Riesz-Lorch}, Riesz and Lorch proved the spectral theorem for self-adjoint operators. 
(Their proof is extend in \cite{Conway} and \cite{Rudin} to normal operators.)
They first established the result for bounded self-adjoint operators, then considered the unbounded case.
Theorem \ref{thm:092} was used to derive the result for the unbounded case from that for the bounded one.
In their argument, the commutativity $AB=BA$ is used a bit indirectly.
Moreover, for some reason or other, the self-adjointness of $B$ is not used.

The second aim of this paper is to give an argument with a direct use of the commutativity $AB=BA$ and the normality of $B$.
To this end, the product of two spectral measures is essential.
Our argument gives a natural way to derive Theorem \ref{thm:065} from the bounded version of it.

This paper is organized as follows.
In Section \ref{sec:decomposition}, we prove Theorem \ref{thm:018}.
In Section \ref{sec:spec thm}, we first introduce product spectral measures, then we prove Theorem \ref{thm:065}, using Theorem \ref{thm:018}.

There are various proofs of the spectral theorem.
For a historical overview of this theorem and its proof, see \cite[Section X.9]{Dunford-Schwartz}.


\section{Proof of Theorem \ref{thm:018}}
\label{sec:decomposition}

Let $T:\HS\to\HS$ be a linear operator. 
We denote by $\D(T)$ and $\RA(T)$ the domain and the range of $T$, respectively. 
For two operators $T$ and $S$, we write $T\subset S$ if $\D(T)\subset\D(S)$ and $Tx=Sx$ holds for all $x\in\D(T)$. 
If $T\subset S$ and $T\supset S$, then we write $T=S$. 
We say that $T$ is \emph{densely defined} if $\D(T)$ is dense in $\HS$. 
For a densely defined operator $T$, we denote by $T^{*}$ the adjoint operator of $T$. 
We recall the definitions of a normal operator and a self-adjoint operator.

\begin{definition}
Let $T:\HS\to\HS$ be a densely defined operator.
We say that $T$ is \emph{normal} if $T^*T=TT^*$. 
Moreover, we say that $T$ is \emph{self-adjoint} if $T^*=T$.
\end{definition}

Let $I:\HS\to\HS$ be the identity mapping.
For the proof of Theorem \ref{thm:018}, we need the following theorem:

\begin{theorem}
\label{thm:500}
Let $T:\HS\to\HS$ be a closed and densely defined operator. 
Then the following three assertions hold:
\begin{enumerate}
\item[(i)]$T^{*}T:\HS\to\HS$ is self-adjoint.
\item[(ii)]$I+T^{*}T$ is injective and $\RA(I+T^{*}T)=\HS$.
\item[(iii)]$(I+T^{*}T)^{-1}$ belongs to $\BH$ and is self-adjoint.
\end{enumerate}
\end{theorem}
\begin{proof}
See, e.g., \cite[Proposition X.4.2]{Conway}.
\end{proof}

Let us prove Theorem \ref{thm:018}.

\begin{proof}[Proof of Theorem \ref{thm:018}]
Thanks to Theorem \ref{thm:500}, we can define $A,B:\HS\to\HS$ by 
\begin{align*}
A=(I+T^{*}T)^{-1},\quad B=TA.
\end{align*}
Then, $A\in\BH$ from Theorem \ref{thm:500} (iii).
We show $B\in \BH$. 
Since $\RA(A)=\D(I+T^{*}T)=\D(T^{*}T)\subset\D(T)$, we have $\D(B)=\D(A)=\HS$.
Moreover, $B$ is closed since $T$ is closed and $A$ is bounded.
Therefore, by the closed graph theorem, we have $B\in\BH$. 

(i) $A$ is injective from Theorem \ref{thm:500} (ii). 

(ii)\ We first show the following claim:

\begin{cla}
\label{cla:012}
\normalfont
$AT\subset B$.
\end{cla}
\begin{proof}[Proof of Claim \ref{cla:012}]
We have $\D(AT)\subset\D(B)$ since $\D(B)=\HS$. 
Let $x\in\D(AT)$. 
Theorem \ref{thm:500} (ii) implies that there is $y\in\D(I+T^{*}T)$ such that $x=(I+T^{*}T)y$. 
Clearly, $Bx=Ty$. 
On the other hand, since $T$ is normal, $T(I+T^{*}T)=T+(TT^{*})T=T+(T^{*}T)T=(I+T^{*}T)T$, and hence, $ATx=(I+T^{*}T)^{-1}((I+T^{*}T)Ty)=Ty$. 
Combining, we obtain $ATx=Bx$.
\end{proof}

We prove (ii).  
By Claim \ref{cla:012}, we have $T=(A^{-1}A)T=A^{-1}(AT)\subset A^{-1}B$.
It remains to show $\D(T)\supset\D(A^{-1}B)$. 
Let $x\in\D(A^{-1}B)$. 
Then, $Bx\in\D(A^{-1})$.
Theorem \ref{thm:500} (ii) implies that there is $y\in\D(I+T^{*}T)$ such that $x=(I+T^{*}T)y$. 
We have $Ty\in\D(A^{-1})$ since $Bx=Ty$.
Moreover, since $T$ is normal, $\D(A^{-1})=\D(I+T^{*}T)=\D(T^{*}T)=\D(TT^{*})$.
Thus, we have $Ty\in\D(TT^{*})$, and hence, $T^{*}Ty\in\D(T)$.
This and $x=y+T^*Ty$ imply $x\in\D(T)$.

(iii)\ It is enough for us to prove $AB\subset BA$. 
By Claim \ref{cla:012}, we have $AB=A(TA)=(AT)A\subset BA$.

(iv)\ $A$ is self-adjoint from Theorem \ref{thm:500} (iii).
Notice that $AT$ is densely defined since $\D(A)=\HS$ and  $T$ is densely defined.
Thus, we can define $(AT)^{*}$.
Moreover, by Claim \ref{cla:012}, we have $B^{*}\subset(AT)^{*}=T^{*}A$. This and $\D(B^{*})=\HS$ imply
\begin{equation}
\label{eq:806}
B^{*}=T^{*}A.
\end{equation}
By (\ref{eq:806}) and the commutativity (iii), we have $B^{*}B=T^{*}AB=T^{*}BA=(T^{*}T)AA$.
Since $T$ is normal, $(T^{*}T)AA=(TT^{*})AA$.
Moreover, (iii) and the self-adjointness of $A$ imply $AB^{*}=B^{*}A$, and hence, $(TT^{*})AA=TB^{*}A=TAB^{*}=BB^{*}$.
Combining, we conclude that $B^{*}B=BB^{*}$, that is, $B$ is normal. 
Moreover, if $T$ is self-adjoint, then $B^{*}=TA=B$ from (\ref{eq:806}), and hence, we conclude that $B$ is self-adjoint.
\end{proof}


\section{Proof of the spectral theorem}
\label{sec:spec thm}

We quote the spectral theorem for bounded normal operators:
\begin{thm}
\label{thm:073}
Let $T:\HS\to\HS$ be a normal operator. 
If $T$ is bounded, then there is a Borel spectral measure $E$ such that the following equality holds:
\[
T=\int_{\C}z\,E(dz).
\]
\end{thm}
For the proof of Theorem \ref{thm:073}, see, e.g., \cite[Theorem IX.2.2]{Conway}.

The aim of this section is to derive Theorem \ref{thm:065} from Theorem \ref{thm:073} via Theorem \ref{thm:018}, with a direct use of the commutativity $AB=BA$ and the normality of $B$.
To this end, we need the following theorem:

\begin{theorem}
\label{thm:022}
Let $\Omega_1$ and $\Omega_2$ be locally compact Hausdorff topological spaces with countable open basis and let $E_1$ and $E_2$ be Borel spectral measures on $\Omega_1$ and $\Omega_2$, respectively. 
Assume that $E_1$ and $E_2$ act on the same Hilbert space and commute, that is, $E_1(A)E_2(B)=E_2(B)E_1(A)$ holds for any Borel sets $A\subset\Omega_1$ and $B\subset\Omega_2$. 
Then there is a unique Borel spectral measure $E_1\times E_2$ on $\Omega_1\times\Omega_2$ such that
\[
(E_1\times E_2)(A\times B)=E_1(A)E_2(B)
\]
holds for any Borel sets $A\subset\Omega_1$ and $B\subset\Omega_2$. 
\end{theorem}
\begin{proof}
See, e.g., \cite[Theorem 4.10]{Schmudgen}.
\end{proof}

We call the above $E_1\times E_2$ the \emph{product spectral measure} of $E_1$ and $E_2$. 
We can verify the following Fubini-type theorem:

\begin{theorem}
\label{thm:042}
Let $\Omega_1,\Omega_2$ and $E_1,E_2$ be as in Theorem \ref{thm:022} and let $f_1:\Omega_1\to\C,\,f_2:\Omega_2\to\C$ be Borel measurable functions. 
If $f_2$ is bounded, then we have
\[
\left(\int_{\Omega_1}f_1\,dE_1\right)\left(\int_{\Omega_2}f_2\,dE_2\right)=\int_{\Omega_1\times\Omega_2}f_1(\omega_{1})f_2(\omega_{2})\,(E_1\times E_2)(d\omega_{1}d\omega_{2}).
\]
\end{theorem}
\begin{proof}
Let $\pi_1$ and $\pi_2$ be the projections from $\Omega_1\times\Omega_2$ to $\Omega_1$ and $\Omega_2$, respectively.
Then, $\int_{\Omega_i}f_i\,dE_i=\int_{\Omega_1\times\Omega_2}f_i\circ\pi_i\,d(E_1\times E_2)$ holds for $i=1,2$.
Thus, the desired result follows from \cite[Theorem 4.16 (iii)]{Schmudgen}.
\end{proof}

\begin{remark}
For a closable linear operator $T$, we denote by $\overline{T}$ the closure of $T$. 
\cite[Theorem 4.16 (iii)]{Schmudgen} actually states that $\overline{(\int f\,dE)(\int g\,dE)}=\int fg\,dE$ holds, where $E$ is a spectral measure on a measurable space $\Omega$, and $f,g:\Omega\to\C$ are measurable functions (they may be unbounded). 
In Theorem \ref{thm:042}, by the boundedness of $f_2$, the operator $(\int_{\Omega_1}f_1\,dE_1)(\int_{\Omega_2}f_2\,dE_2)$ is closed, and we need not to take closure.
\end{remark}

We are ready to prove Theorem \ref{thm:065}.

\begin{proof}[Proof of Theorem \ref{thm:065}]
Let $A,B\in \BH$ be as in Theorem \ref{thm:018}. 
Since $A$ is self-adjoint and $B$ is normal, from Theorem \ref{thm:073}, there are Borel spectral measures $E_1$ and $E_2$ on $\R$ and $\C$, respectively, satisfying
\[
A=\int_{\R}t\,E_1(dt),\quad B=\int_{\C}w\,E_2(dw).
\]
Moreover, since $B$ is bounded, there is a compact set $K\subset\C$ such that $B=\int_{K}w\,E_2(dw)$.
We define $f:\R\to\R$ by $f(0)=0,\,f(t)=1/t\>(t\neq0)$. 
It is easy to show
\[
A^{-1}=\int_{\R}f(t)\,E_1(dt).
\]
The commutativity $AB=BA$ implies that $E_1$ and $E_2$ commute. 
Therefore, from Theorem \ref{thm:022}, the product spectral measure $E_1\times E_2$ on $\R\times K$ exists. 
By Theorem \ref{thm:042}, we have
\[
T=\int_{\R\times K}f(t)w\,(E_1\times E_2)(dt\,dw)=\int_{\C}z\,E(dz),
\]
where $E$ denotes the push forward of $E_1\times E_2$ by the map $\R\times K\ni(t,w)\mapsto f(t)w\in\C$. 
We obtain the desired result. 
\end{proof}

\centerline{\scshape Acknowledgments}
\medskip
The author would like to thank Mishio Kawashita (Hiroshima University), Koichiro Iwata (Hiroshima University) and Masao Hirokawa (Kyushu University) for helpful advices and comments. 
The author is supported by FY2019 Hiroshima University Grant-in-Aid for Exploratory Research (The researcher support of young Scientists).

\end{document}